\documentclass[reqno]{amsart}
\usepackage{amscd}
\usepackage[dvips]{graphics}
\usepackage{color}

\def\beq{\begin{equation}}
\def\eeq{\end{equation}}
\def\ba{\begin{array}}
\def\ea{\end{array}}

\def\R{\mathbb R}

\def\Z{\mathbb Z}

\newcommand{\rmnote}[1]{}

\numberwithin{equation}{section}
\newenvironment{abs}{\textbf{Abstract}\mbox{  }}{ }
\newenvironment{key words}{\textbf{Keywords}\mbox{  }}{ }
\newtheorem{theorem}{Theorem}[section]

\newtheorem{corollary}[theorem]{\textbf{Corollary}}
\newtheorem{proposition}[theorem]{\textbf{Proposition}}
\newtheorem{lemma}[theorem]{Lemma}
\renewenvironment{proof}{\noindent{\textbf{Proof.}}}{\hfill$\Box$}
\theoremstyle{remark}

\theoremstyle{plain}
\begin{document}
\title{\textbf{Overdetermined elliptic problems in onduloid-type domains with general nonlinearities}}

\author{David Ruiz}
\address{(D.~Ruiz)
Departamento de An\'alisis matem\'atico, Universidad de Granada,
Campus Fuentenueva, 18071 Granada, Spain} \email{daruiz@ugr.es}

\author{Pieralberto Sicbaldi}
\address{(P.~Sicbaldi)
Departamento de An\'alisis matem\'atico,
Universidad de Granada,
Campus Fuentenueva,
18071 Granada,
Spain \& Aix Marseille Universit\'e - CNRS, Centrale Marseille - I2M, Marseille, France}
\email{pieralberto@ugr.es}

\author{Jing Wu}
\address{(J.~Wu)
Departamento de An\'alisis matem\'atico, Universidad de Granada,
Campus Fuentenueva, 18071 Granada, Spain} \email{jingwulx@correo.ugr.es}

\thanks{D. R. has been supported by the FEDER-MINECO Grant PGC2018-096422-B-I00 and by J. Andalucia (FQM-116). P. S has been supported by the FEDER-MINECO Grant MTM2017-89677-P and by J. Andalucia (P18-FR-4049 and A-FQM-139-UGR18). J. W. has been supported by the China Scholarship Council (CSC201906290013) and by J. Andalucia (FQM-116).}

\thanks{ The authors acknowledge financial support from the Spanish Ministry of
Science and Innovation (MICINN), through the \emph{Severo Ochoa and Mar\'{\i}a de
Maeztu Programme for Centres and Unities of Excellence}
(CEX2020-001105-M)}

\maketitle

\maketitle

\noindent
\begin{abs}
In this paper, we prove the existence of nontrivial unbounded domains $\Omega\subset\mathbb{R}^{n+1},n\geq1$, bifurcating from the straight cylinder $B\times\mathbb{R}$ (where $B$ is the unit ball of $\mathbb{R}^n$), such that the overdetermined elliptic problem
\begin{equation*}
  \begin{cases}
  \Delta u +f(u)=0 &\mbox{in $\Omega$, }\\
  u=0 &\mbox{on $\partial\Omega$, }\\
  \partial_{\nu} u=\mbox{constant} &\mbox{on $\partial\Omega$, }
  \end{cases}
\end{equation*}
has a positive bounded solution. We will prove such result for a very general class of functions $f: [0, +\infty) \to \mathbb{R}$. Roughly speaking, we only ask that the Dirichlet problem in $B$ admits a nondegenerate solution. The proof uses a local bifurcation argument.
\end{abs}\\
\begin{key words}: Overdetermined boundary conditions; semilinear elliptic problems; bifurcation theory.
\end{key words}
 \indent


\section{Introduction and statement of the main result}
\label{Section 1}

This paper is devoted to the existence of new solutions of a semilinear overdetermined elliptic problem in the form
\begin{gather}\label{eq11}
\begin{cases}
  \Delta u +f(u)=0 &\mbox{in $\Omega$, } \\
  u>0 &\mbox{in $\Omega$, } \\
  u=0 &\mbox{on $\partial\Omega$, }\\
  \partial_{\nu} u=\mbox{constant} &\mbox{on $\partial\Omega$, }
  \end{cases}
\end{gather}
where $\Omega$ is a domain of $\mathbb{R}^{n+1},n\geq1, f: [0, +\infty) \to \mathbb{R}$ is a $C^{1,\alpha}$ function and $\nu$ stands for the exterior normal unit vector on $\partial\Omega$.

\medskip

A classical result by J. Serrin ~\cite{S71} (see also \cite{PS07}) states that the existence of a solution to the overdetermined problem (\ref{eq11}) in a bounded domain $\Omega$ implies that $\Omega$ is a ball and the solution $u$ is radially symmetric. The proof uses the reflection principle in the spirit of the result of Alexandrov for compact embedded CMC hypersurfaces (\cite{ale}). The result of Serrin has applications in various mathematical and physical problems, such as isoperimetric inequalities, spectral geometry and hydrodynamics (see, e.g., ~\cite{BK06,S02,S56} for the details).

\medskip

The case when the domain $\Omega$ is supposed to be unbounded is also of interest. Indeed, overdetermined boundary conditions appear in free boundary problems if the variational structure imposes suitable conditions on the separation interface (see, e.g. ~\cite{AC81,CJK04}). In this context, the methods used to study the regularity of the solutions of a free boundary problem are based on blow-up techniques, and then lead to the study of a semilinear overdetermined elliptic problem in an unbounded domain. In this framework, H. Berestycki, L. Caffarelli and L. Nirenberg ~\cite{BCN97} considered the problem (\ref{eq11}) in unbounded domains $\Omega$ and they proposed the following:

\medskip

BCN Conjecture. Assume that $\Omega$ is a smooth domain such that $\mathbb{R}^{n}\backslash \Omega$ is connected. Then the existence of a bounded solution to the problem (\ref{eq11}) for some Lipschitz function $f$ implies that $\Omega$ is either a ball, a half-space, a generalized cylinder $B^{k}\times\mathbb{R}^{n-k}$ ($B^{k}$ is a ball in $\mathbb{R}^{k}$), or the complement of one of them.

\medskip

Such conjecture was motivated first by the result of Serrin for bounded domains, and for unbounded domains by some rigidity results obtained in epigraphs (\cite{BCN97}) and exterior domains (\cite{AB98, R97}). The BCN Conjecture actually has motivated various interesting works giving an affirmative answer for some classes of overdetermined elliptic problems. Let us now briefly describe some of such results. In \cite{FV101} the authors get an affirmative answer under the hypothesis that $\Omega$ is an epigraph of $\R^2$ or $\R^3$ and the function $u$ satisfies some natural assumptions. Moreover, in \cite{RS13} the BCN conjecture is proved in the plane for some classes of nonlinearities $f$. In the harmonic case $f=0$ a complete classification of solutions to the problem \eqref{eq11} in the plane has been given in \cite{T13}. Moreover, the work \cite{RRS17}, proves the BCN conjecture in dimension $2$ if $\partial\Omega$ is connected and unbounded.

\medskip

It turns out, however, that the BCN conjecture does not hold true in general. The first counterexample was given in \cite{S10}, where the second author constructed domains obtained by periodic perturbation of the straight cylinder $B^{n}\times\mathbb{R}$ for which the problem (\ref{eq11}) with $f(u)=\lambda u, \lambda>0$, admits a solution. More precisely, such domains, as shown in \cite{SS12}, belong to a 1-parameter family $\{\Omega_s\}_{s \in (-\epsilon,\epsilon)}$ and are given by
\[
\Omega_s \,=\,
\left\{
(x,t) \in \mathbb{R}^{n} \times \mathbb{R}  \, : \  |x| <
1+s \cos \left( \frac{2\pi}{T_s}\,t \right) + O(s^2)
\right\}
\]
where $\epsilon$ is a small constant, $T_s = T_0 + O(s)$ and $T_0$ depends only on the dimension $n$. This result reinforces the analogy between domains that allow a solution of \eqref{eq11} and CMC surfaces, as the domain $\Omega_s$ can be put in correspondence to the onduloid (or Delaunay surface).

In \cite{FMW17} the same kind of result is proved in the case $f\equiv1$. In \cite{DPW15} similar solutions are found for the Allen-Cahn nonlinearity $f(u)=u-u^3$, but in domains that are perturbations of a dilated straight cylinder, i.e. perturbations of $(\epsilon^{-1}\,B^{n})\times\mathbb{R}$ for $\epsilon$ small, or more in general domains that are perturbations of a dilation of the region contained in an onduloid. The result in \cite{S10} has been generalized in some Riemannian manifolds with symmetry in \cite{MS16}

\medskip

The aim of this paper is to perform such a construction under somewhat minimal assumptions on the nonlinearity $f(u)$. It is clear that a mandatory assumption is the existence of a solution of the Dirichlet problem in the unit ball $B$ in $\R^n$. For technical reasons we need the normal derivative at the boundary to be nonzero, which is a typical situation in overdetermined semilinear elliptic problems. Hence, we will consider the following hypothesis:

\medskip

\textbf{Assumption 1:} There exists a positive solution ${\phi}_{1}\in C^{2,\alpha}(\overline{B})$ of the problem
\begin{equation}\label{eq214}
\begin{cases}
\Delta{\phi}_1+f({\phi}_1)=0 &\mbox{in $B$, }\\
{\phi}_1=0 &\mbox{on $\partial B$,}
\end{cases}
\end{equation}
with $\partial_\nu \phi_1 (x) \neq 0$ for $x \in \partial B$, where $\nu$ denotes the exterior unit vector normal to $\partial B$.

\medskip Observe that by \cite{gnn}, any solution $\phi_1$ of \eqref{eq214} needs to be a radially symmetric function. For technical reasons, we need to assume also that the linearized operator associated to problem \eqref{eq11} at $\phi_1$ is non-degenerate (in a radially symmetric setting). This is a rather natural assumption if one intends to use a perturbation argument. Precisely, our second assumption is:

\medskip


\textbf{Assumption 2:} Define the linearized operator $L_{D}:C^{2,\alpha}_{0,r}(B)\rightarrow C^{0,\alpha}_{r}(B)$ by
\begin{equation*}
  L_{D}(\phi)=\Delta\phi+f'({\phi}_{1})\phi\,,
\end{equation*}
where $C^{2,\alpha}_{0,r}(B)$ and $C^{0,\alpha}_{r}(B)$ denote the spaces of radial functions in $C^{2,\alpha}_{0}(B)$ and $C^{0,\alpha}(B)$ respectively.
We assume that
the linearized operator $L_{D}$ is non-degenerate; in other words, if $L_D(\phi)\equiv0$ then $\phi \equiv 0$.

\medskip

We are now in position to state our main result:
\begin{theorem} \label{Th11}
If $n \geq 1$, $f: [0,+\infty) \to \R$ is $C^{1,\alpha}$ and Assumptions 1 and 2 hold, then there exist a positive number $T_*$ and a continuous curve
$$\begin{array}{ccc}
(-\epsilon, \epsilon) & \to & C^{2,\alpha}(\mathbb{R} / \mathbb{Z}) \times \mathbb{R}\\
s & \mapsto & (v_s,T_s)
\end{array}$$
for some $\epsilon$ small, with $v_s=0$ if and only if $s=0$. Moreover $T_0=T_*$ and
the overdetermined problem \eqref{eq11} has a solution in the domain
\[
\Omega_s \,=\,
\left\{ (x,t) \in \mathbb{R}^{n} \times \mathbb{R}  \, :  \ |x| < 1 + v_s\left(\frac{t}{T_s}\right) \right\}\,.
\]
The solution $u= u_s$ of problem \eqref{eq11} is $T_s$-periodic in the variable $t$ and hence bounded. Moreover,
\[
\int_0^1v_{s}(t)\,dt=0.
\]
\end{theorem}

\medskip

Let us point out that Assumptions 1 and 2 hold for example in the following cases (among many others):
\begin{itemize}
  \item[(1)] If $f(0)> 0$ and $f'(s)<\lambda_{1}$ for any $s\in(0,+\infty)$, where $\lambda_{1}$ is the first eigenvalue of the Dirichlet Laplacian in the unit ball of $\mathbb{R}^n$. In this case a positive solution can be found (for instance, extending $f(s)=f(0)$ if $s <0$ and minimizing the corresponding Euler-Lagrange functional) and the operator $L_D$ has only positive eigenvalues.
   \item[(2)] If $f(u)=u^{p} - u$, $1<p<\frac{n+2}{n-2}$ if $n > 2$, $p>1$ if $n=2$. In this case the existence of a solution is well known, and it is a mountain-pass solution. As a consequence, $L_D$ has a negative eigenvalue. By the analysis of \cite{kwong}, all other eigenvalues are strictly positive.
   \item[(3)] If $f(u)= \lambda e^u$ and $\lambda \in (0, \lambda^*)$ for some $\lambda^*>0$ that receives the name of extremal value. In this case $\phi_1$ is the so-called minimal solution and $L_D$ has only positive eigenvalues (see for instance \cite{dupaigne}).
\end{itemize}

In particular, (1) holds when $f \equiv 1$, and we recover in this way the result in \cite{FMW17}. On the other hand, when $f(u) = \lambda\, u$ for some $\lambda>0$, Assumption 1 implies that $\lambda$ is the first eigenvalue of the Dirichlet Laplacian in the unit ball of $\mathbb{R}^{n}$, but then Assumption 2 is clearly not satisfied. Hence, our theorem is complementary to the results in \cite{SS12, S10}

\medskip

Theorem \ref{Th11} is a bifurcation result in the spirit of \cite{S10}, see also \cite{SS12, FMW17}. In sum, one can reformulate the existence of solutions to \eqref{eq11} as the zeroes of a nonlinear Dirichlet-to-Neumann operator, and the Crandall-Rabinowitz Theorem is used to conclude local bifurcation. But here the situation is more involved because of the general term $f(u)$. In fact, the operator $L_D$ may have negative eigenvalues, and the bifurcation argument requires a finer spectral analysis. In particular, the Dirichlet-to-Neumann operator can be built only for certain values of $T$, which are related to the nondegeneracy of the Dirichlet problem in the cylinder. We are able to show the existence of a bifurcation branch by taking advantage of the ideas of \cite{RRS20}, where the linearized problem has a negative eigenvalue.

\medskip

The rest of the paper is organized as follows. In Section 2 and Section 3 we give some basic preliminary results on the Dirichlet problem in a ball and in a cylinder respectively. These results will be needed for the construction of the nonlinear Dirichlet-to-Neumann operator, which will be performed in Section 4. In Section 5 we will compute the linearization of this operator. Section 6 is devoted to study the properties of the linearized operator computed in Section 5. With all those ingredients we can use a local bifurcation argument to prove Theorem ~\ref{Th11}; this will be accomplished in Section 7.

\numberwithin{equation}{section}\section{Some preliminaries about related linear problems in the ball}
\label{Section 2}
Let $B$ be the unit ball in $\mathbb{R}^n$ centered at the origin. It will be useful to define the following H\"{o}lder spaces:
\[C_{r}^{k,\alpha}(B)=\big\{\phi\in C^{k,\alpha}(B): \phi(x)=\phi(|x|),~x\in B\big\}\,,\]
\[C_{0,r}^{k,\alpha}(B)=\big\{\phi\in C_{0}^{k,\alpha}(B): \phi(x)=\phi(|x|),~x\in B\big\}\,.\]
We also define the following Sobolev spaces:
 \[H_{r}^{1}(B)=\big\{\phi\in H^{1}(B):\phi(x)=\phi(|x|),~x\in B\big\}\,,\]
 \[H_{0,r}^{1}(B)=\big\{\phi\in H_{0}^{1}(B): \phi(x)=\phi(|x|),~x\in B\big\}\,.\]

We will write $r=|x|$, and for functions $\phi$ in such spaces we will use both notations $\phi(x)$ and $\phi(r)$ according to the computations. We recall that $\phi_1$ is a radial solution of \eqref{eq214}, that is,

 \begin{equation}\label{eq369}
\begin{cases}
\phi_{1}''(r)+\frac{n-1}{r}\phi_{1}'(r)+f(\phi_{1}(r))=0 &\mbox{in $(0,1]$ },\\
\phi_{1}(1)=0,\quad\phi_{1}'(0)=0.
\end{cases}
\end{equation}

The operator $L_D$ defined in Assumption 2 has a diverging sequence of eigenvalues $\gamma_{D_j}$, hence there are only a finite number $l$ of them which are negative: $$\gamma_{D_1}<\gamma_{D_2}<\cdots<\gamma_{D_l}<0, \ \gamma_{D_{l+1}}>0.$$

Next result is rather standard, we include it here for the sake of completeness.
\begin{lemma} \label{le300}
The eigenvalues $\gamma_{D_j}$ are all simple.
 \end{lemma}
        \begin{proof} Assume that $\psi_{1},\psi_{2}$ are two nontrivial eigenfunctions corresponding to $\gamma_{D_j}$, i.e.
         \[L_{D}(\psi_{i})+\gamma_{D_j}\psi_{i}=0\,,\,i=1,2\,.\]
         Let us choose $k_{1},k_{2}\in\mathbb{R},(k_{1},k_{2})\neq(0,0),$ such that
         \[k_{1}\psi'_{1}(1)+k_{2}\psi'_{2}(1)=0\,.\]
         If ${\psi}=k_{1}\psi_{1}+k_{2}\psi_{2},$ we have
         \[L_{D}({\psi})+\gamma_{D_j}{\psi}=0, \ \psi(1)=0, \ {\psi}'(1)=0\,.\]
         By the uniqueness of the solution of the Cauchy problem for ODE, ${\psi}\equiv0.$ That is,  $\psi_{1},\psi_{2}$ are linearly dependent. \end{proof}

\medskip

We denote by $z_{j}\in C_{0,r}^{2,\alpha}(B)$ the eigenfunction with eigenvalue  $\gamma_{D_j}$, i.e.
\begin{equation} \label{z}
\begin{cases}
\Delta z_{j}+f'(\phi_{1})z_{j}+\gamma_{D_j}z_{j}=0 &\mbox{in $B$\,, }\\
z_{j}=0 &\mbox{on $\partial B$\,,}\\
\end{cases}
\end{equation}
normalized by $\|z_{j}\|_{L^2}=1$.

\medskip

As is well known, the operator $L_D$ is related to the quadratic form
\[Q_{D}:H_{0,r}^1(B) \rightarrow\mathbb{R}, \ \ Q_{D}(\phi):=\int_{B}\big(|\nabla\phi|^{2}-f'({\phi}_{1})\phi^{2}\big)\,.\]
For instance, the first eigenvalue of $L_{D}$ is given by
\begin{equation} \label{gammaD1}\gamma_{D_1}=\inf\left\{Q_{D}(\phi):\|\phi\|_{L^{2}(B)}=1\right\}.\end{equation}

\medskip

In Section \ref{Section 7} our computations will involve another quadratic form $Q$ defined as \[ Q: H^1_r(B) \to \R, \ \ Q(\psi):=\int_{B}\big(|\nabla\psi|^{2}-f'(\phi_{1})\psi^{2}\big)+c\,\omega_{n}\psi(1)^2\,,\]
where $\omega_n$ is the area of $\mathbb{S}^{n-1}$ and
\begin{equation} \label{defc}  c= -\frac{\phi_1''(1)}{\phi_1'(1)} = n-1 + \frac{f(0)}{\phi_1'(1)}\,.\end{equation}
To get the last equality we used \eqref{eq369}. Observe that
\begin{equation} \label{quadratics} Q|_{H_{0,r}^1(B)}= Q_D\,. \end{equation}
Analogously we can define
\begin{equation} \label{gamma1} \gamma_{1}=\inf\left\{Q(\phi):\|\phi\|_{L^{2}(B)}=1\right\}\,.\end{equation}
It is rather standard to show that $\gamma_1$ is achieved by a solution $\psi_1$ of the problem:
\begin{equation}\label{eqlineal}
\begin{cases}
\Delta{\psi}_1+f'({\phi}_1) \psi_1 + \gamma_1 \psi_1=0 &\mbox{in $B$, }\\
\partial_{\nu} \psi_1(x) +c \,  \psi_1(x) =0 &\mbox{on $\partial B$\,.}
\end{cases}
\end{equation}
As in Lemma \ref{le300} one can show that $\gamma_1$ is simple, so $\psi_1$ is uniquely determined up to a sign.

\medskip

We finish this section with an estimate of the eigenvalue $\gamma_1$.

\begin{lemma} \label{lema2.2} There holds: $\gamma_1 < \min \{0, \gamma_{D_1} \}$.
\end{lemma}
%

\begin{proof}
We first show that $\gamma_1<0$; for this it suffices to find $\psi \in H^1_r(B)$ such that $Q(\psi)<0$. Since $Q$ is considered among radially symmetric functions, we can write the quadratic form as
\begin{align*}
    Q(\psi)&=\int_{B}\big[|\nabla\psi|^{2}-f'(\phi_{1})\psi^{2}\big]+c \, \omega_{n}\psi(1)^{2}\\
    &=\omega_{n}\int_0^1r^{n-1}\big[\psi'(r)^{2}-f'(\phi_{1})\psi(r)^{2}\big]dr+c \, \omega_{n}\psi(1)^{2}.
\end{align*}
Now we compute the derivative in \eqref{eq369} to obtain:
  \begin{equation}\label{eq368}
  \phi_{1}'''(r)+\frac{n-1}{r}\phi_{1}''(r)-\frac{n-1}{r^{2}}\phi_{1}'(r)+f'(\phi_{1})\phi_{1}'(r)=0.
\end{equation}
If we multiply the equation (\ref{eq368}) by $r^{n-1}\phi_{1}'(r)$ and integrate, we obtain
\begin{equation*}
\int_0^1r^{n-1} \left[ \phi_{1}''(r)^{2}-f'(\phi_{1})\phi_{1}'(r)^{2}\right]dr=\phi_{1}'(1)\phi_{1}''(1)-(n-1)\int_0^1r^{n-3}\phi_{1}'(r)^{2}dr.
\end{equation*}
This last equality comes from the computation:
\begin{align*}
&\int_0^1r^{n-1}\phi_{1}'''(r)\phi_{1}'(r)dr=\int_0^1r^{n-1}\phi_{1}'(r)d\phi_{1}''(r)\\
&=r^{n-1}\phi_{1}'(r)\phi_{1}''(r)\big|^{1}_{0}-\int_0^1\phi_{1}''(r)d\big(r^{n-1}\phi_{1}'(r)\big)\\
&=\phi_{1}'(1)\phi_{1}''(1)-\int_0^1r^{n-1}\phi_{1}''(r)^{2}dr-(n-1)\int_0^1r^{n-2}\phi_{1}'(r)\phi_{1}''(r)dr.
\end{align*}
We can take the test function $\phi_{1}'(r)\in H_{r}(B)$ obtaining:
  \begin{align*}
  Q(\phi_{1}'(r)) &=\omega_{n}\int_0^1r^{n-1}\big[\phi_{1}''(r)^{2}-f'(\phi_{1})\phi_{1}'(r)^{2}\big]dr+c \, \omega_{n}\phi_{1}'(1)^{2}\\
     &=-(n-1)\omega_{n}\int_0^1r^{n-3}\phi_{1}'(r)^{2}dr.
\end{align*}

If $n>1$, we have already found a radial function $\psi$ such that $Q(\psi)<0$. In the case $n=1$, $Q(\phi_1')=0$, and indeed $\phi_1'$ is a solution of the linearized problem:

\begin{equation}\label{eqlineal1D}
\begin{cases}
{\psi}''+f'({\phi}_1) \psi =0 &\mbox{in $[-1,1]$, }\\
\psi'(1) +c \,  \psi(1) =0, & -\psi'(-1) +c \,  \psi(-1) =0.
\end{cases}
\end{equation}

However this solution cannot correspond to the first eigenvalue $\gamma_1$ since $\phi_1'$ changes sign in $[-1,1]$. As a consequence, $\gamma_1$ is negative.

\medskip We now show that $\gamma_1 < \gamma_{D1}$.  From \eqref{gammaD1}, \eqref{gamma1} and \eqref{quadratics},
we have immediately $\gamma_1 \leq \gamma_{D_1}$. Assume, reasoning by contradiction, that $\gamma_1 = \gamma_{D_1}$. Hence the minimizer $z_1\in H_{0,r}^1(B)$ works also for the minimizing problem defining $\gamma_1$. In particular, $z_1$ solves \eqref{eqlineal}, and its boundary condition implies that $\partial_\nu z_1(x)=0$ for $x \in \partial B$. Summing up, $z_1$ solves:
$$ \begin{cases}
\Delta{z}_1+f'({\phi}_1) z_1 + \gamma_1 z_1=0 &\mbox{in $B$, }\\
z_1(x) =0 &\mbox{on $\partial B$,} \\
\partial_{\nu} z_1(x) =0 &\mbox{on $\partial B$.}
\end{cases}$$
But, by the uniqueness of the Cauchy problem for ODE we conclude that $z_1=0$, a contradiction.
\end{proof}

\numberwithin{equation}{section}\section{Eigenvalue estimates for related linear problems in the cylinder  }
\label{Section 3}

As commented in the introduction, the construction of the Neumann-to-Dirichlet operator (which will be made in next section) can be performed only if the Dirichlet problem in the cylinder is not degenerate. The main purpose of this section is to study this question. We will show that we have nondegeneracy for all $T \in (0, \overline{T})$, for some specific value of $\overline{T}$. Hence the rest of the computations of the next sections will always require $T \in (0, \overline{T})$.

Let us consider the Dirichlet problem for the linearized equation in a straight cylinder for periodic functions, namely,
\begin{equation}\label{cilindro}
\begin{cases}
\Delta{\psi}+f'({\phi}_1) \psi= 0 &\mbox{in $B \times \R$, }\\
 \psi(x) =0 &\mbox{on $(\partial B) \times \R$,}
\end{cases}
\end{equation}
where $\psi(x,t)$ is $T$-periodic in the variable $t$. Define:
\[
C^T_1 = B \times \mathbb{R}/T\mathbb{Z}.
\]
Hence \eqref{cilindro} is just the linearization of the problem:
\begin{equation}\label{eq2201}
  \begin{cases}
  \Delta\phi+f(\phi)=0 &\mbox{in }C^T_1,\\
  \phi=0 & \mbox{on } \partial C^T_1.
  \end{cases}
\end{equation}

We define the following H\"{o}lder spaces of radial functions:
 \[C_{r}^{k,\alpha}(C^T_1)=\big\{\phi\in C^{k,\alpha}(C^T_1): \phi(x,t)=\phi(|x|,t),~(x,t)\in C^T_1\big\},\]
 \[C_{0,r}^{k,\alpha}(C^T_1)=\big\{\phi\in C_{0}^{k,\alpha}(C^T_1): \phi(x,t)=\phi(|x|,t),~(x,t)\in C^T_1\big\}.\]
We also define the following Sobolev spaces:
\[H_{r}^{1}(C^T_1)=\big\{\phi\in H^{1}(C^T_1):\phi(x,t)=\phi(|x|,t),~(x,t)\in C^T_1\big\}\,,\]
\[H_{0,r}^{1}(C^T_1)=\big\{\phi\in H_{0}^{1}(C^T_1): \phi(x,t)=\phi(|x|,t),~(x,t)\in C^T_1\big\}\,.\]
For functions in such spaces sometimes we will write $\phi(r)$ and $\phi(r,t)$ instead of respectively $\phi(x)$ and $\phi(x,t)$, with $r=|x|$. The reader will understand in each case if we refer to the variable $x$ or $r$.

\medskip

If ${\phi}_{1}$ is the solution of Problem \eqref{eq214}, then the function $\phi_{1}(x,t)={\phi}_{1}(x)$ (we use a natural abuse of notation) solves \eqref{eq2201}. Define the linearized operator $L_{D}^{T}:C_{0,r}^{2,\alpha}(C^T_1)\rightarrow C_{r}^{\alpha}(C^T_1)$ (associated to Problem \eqref{eq2201}) by
\begin{equation*}
  L_{D}^{T}(\phi)=\Delta \phi+f'(\phi_{1})\phi,
\end{equation*}
and consider the eigenvalue problem
\[ L_{D}^{T}(\phi)+\tau\phi=0.\]
Then the functions $z_{j}(x,t)=z_{j}(x)$ from \eqref{z} solve the problem 
\begin{equation}\label{eq202}
  \begin{cases}
  \Delta z_{j}+f'(\phi_{1})z_{j}+\tau_{j}z_{j}=0 &\mbox{in $C^T_1$, }\\
  z_{j}=0 &\mbox{on $\partial C^T_1$.}\\
  \end{cases}
\end{equation}
Let us define the quadratic form $Q_{D}^{T}: H_{0,r}^1(C^T_1) \rightarrow\mathbb{R}$ related to $L_{D}^{T}$,
\begin{equation*}
 Q_{D}^{T}(\psi):=\int_{C^T_1}\big(|\nabla \psi|^2-f'(\phi_{1})\psi^2 \big).
\end{equation*}
We will also need to study the quadratic form $Q^{T}: H_{r}^1(C^T_1) \rightarrow\mathbb{R}$,
\begin{equation}\label{eq26bis}
Q^{T}(\psi):=\int_{C^T_1}\big(|\nabla \psi|^2-f'(\phi_{1})\psi^2 \big) + c \int_{\partial C^T_1} \psi^2.
\end{equation}

The main result of this section is next proposition, where we study the behavior of these quadratic forms:

\begin{proposition} \label{Pr300}
Define:
$$
\alpha = \inf \left \{ Q_D^{T} (\psi):\ \psi \in H_{0,r}^1(C^T_1),\ \| \psi \|_{L^2}=1,\ \int_{C^T_1} \psi \, z_j=0,\ j=1, \dots l. \right \}\,,
$$
and
$$
\beta = \inf \left \{ Q^{T} (\psi):\ \psi \in H_r^1(C^T_1),\ \| \psi \|_{L^2}=1,\ \int_{\partial C^T_1} \psi=0,\ \int_{C^T_1} \psi \, z_j=0,\ j=1, \dots l. \right \}\,.
$$
Then
\begin{equation}\label{alpha} \alpha = \min\left\{ \gamma_{D_{l+1}}, \gamma_{D_1} + \frac{4\pi^2}{T^2}\right\},\end{equation}
and
\begin{equation} \label{beta} \beta = \min\left\{ \gamma_{D_{l+1}}, \gamma_{1} + \frac{4\pi^2}{T^2}\right\}\,.\end{equation}
Moreover, those infima are achieved. If $\gamma_{1} + \frac{4\pi^2}{T^2}<  \gamma_{D_{l+1}}$, the minimizer is equal to $$\psi_1(x) \cos\left(\frac{2 \pi}{T} \left(t+\delta\right)\right),$$ where $\psi_1$ is the minimizer for \eqref{gamma1} and $\delta \in [0,1]$.

\end{proposition}

\begin{proof} We prove the result for $\beta$; the result for $\alpha$ is analogous. First, it is rather standard to show that $\beta$ is achieved by a function $\psi$. By the Lagrange multiplier rule, there exist $\theta_{1},\theta_{2}$ and $\zeta_1, \dots \zeta_l$ real numbers so that for any $\rho\in H_{r}^{1}(C^T_1),$
	\[\int_{C^T_1}\big(\nabla\psi\nabla\rho-f'(\phi_{1})\psi\rho + \rho \sum_{i=1}^{l} \zeta_i z_i  + \theta_1 \psi \rho \big ) =\int_{\partial C^T_1}\rho(\theta_{2}+c \psi).\]	
By choosing $\rho = z_j $ we conclude that $\zeta_j=0$. If we now take $\rho= \psi$, we obtain that $\theta_1 = \beta$. Hence $\psi$ is a solution of the equation
\begin{equation*}
\Delta \psi +f'(\phi_{1}) \psi + \beta \psi =0  \ \mbox{\ in }C^T_1.
\end{equation*}

Define now:
\[
\bar{\psi}(x)=\int_0^T\psi(x,t)dt\,.
\]
It is immediate that
\begin{equation} \label{ortho} \int_{B} \bar{\psi} \, z_j = \int_{C^T_1} \psi \,  z_j =0, \ j=1, \dots ,l\,.\end{equation}
By direct computation
  \[\begin{aligned}
    \Delta_x\bar{\psi}&=\int_0^T\Delta_{x}\psi(x,t)dt\\&=\int_0^T\Delta\psi(x,t)dt-\int_0^T\psi_{tt}(x,t)dt\\
    &=\int_0^T\Delta\psi(x,t)dt-(\psi_{t}(x,T)-\psi_{t}(x,0))\\
     &=\int_0^T\Delta\psi(x,t)dt\\&=-\int_0^T (f'(\phi_{1}) + \beta )\psi(x,t)dt\\&=-(f'(\phi_{1}) + \beta) \bar{\psi}.
         \end{aligned}\]
 As a consequence, we have that $\bar{\psi}$ solves the problem
 \begin{equation*}
  \begin{cases}
  \Delta\bar{\psi}+f'(\phi_{1})\bar{\psi}+ \beta \bar{\psi} =0 &\mbox{in $B$,}\\
  \bar{\psi}=0 &\mbox{on $\partial B$.}\\
  \end{cases}
\end{equation*}
Taking into account \eqref{ortho}, there are two cases: either $\beta = \gamma_{D_k}$, $k \geq l+1$, or $\bar{\psi}=0$. In the first case, by plugging $z_{l+1}$ in the definition of $\beta$, we conclude that $k= l +1$.
In the second case we have,
\[
\int_0^T\psi(x,t)dt=0\,  \,\, \forall x \in B.
\]
Hence we can use the Poincar\'{e}-Wirtinger inequality for periodic functions to estimate:
\[ \frac{4 \pi^2}{T^2} \int_0^T\psi^{2}dt \leq \int_0^T\psi_{t}^{2}dt \, .\]
Then, recalling \eqref{gammaD1},
\[\begin{aligned}
    \beta= Q^{T}(\psi)
    &=\int_0^T \left( \int_{B} \big(|\nabla_{x}\psi|^{2}-f'(\phi_{1})\psi^{2}\big) + c \omega_n \psi(1,t)^2 \right) dt +\int_{B} \int_0^T |\psi_{t}|^{2}dt\\
    & \geq \gamma_{1} \int_0^T dt\int_{B} \psi^2 +\frac{4 \pi^2}{T^2}  \int_{B} \int_0^T {{\psi}}^{2}dt\\
   &=\left(\gamma_{1}+\frac{4 \pi^2}{T^2}\right)\int_{C^T_1}\psi^{2} = \gamma_{1}+ \frac{4 \pi^2}{T^2}\,.
         \end{aligned}\]
Moreover, the above inequalities are equalities only if $\psi(x,t)$ is proportional to $\psi_1(x)\cos\left(\frac{2 \pi}{T} (t+\delta)\right)$.
%
%
\end{proof}

\medskip

As a consequence, we can state the following:

\begin{corollary} \label{cor}
	Define $\overline{T}$ as:
	\begin{equation}\label{eq266}
	\overline{T}= \left \{\begin{array}{ll}  \frac{2 \pi}{\sqrt{-\gamma_{D_1}}} & \mbox{ if } \gamma_{D_1}<0\,, \\ +\infty  & \mbox{ if } \gamma_{D_1}>0\,. \end{array} \right.
	\end{equation}
	Then, for any $T \in (0, \overline{T})$, we have that $Q_{D}^{T}(\psi)>0$ for any $\psi\in H_{0,r}^1(C^T_1)
	$ satisfying the orthogonality conditions:
	\[
	\int_{C^T_1}\psi z_{j}=0\,, \,j=1,2,\cdots,l\,.
	\]
	As a consequence, $L_D^{T}$ is nondegenerate for any $T \in (0,\overline{T})$.
\end{corollary}

\medskip


\numberwithin{equation}{section}\section{Perturbations of the cylinder and formulation of the problem}
\label{Section 4}


The main purpose of this section is to build a nonlinear Dirichlet-to-Neumann operator $G$ associated to \eqref{eq11} for any $T \in (0, \overline{T})$.

Given a positive number $T$ and a $C^{2,\alpha}$ function $v:\mathbb{R}/\mathbb{Z}\rightarrow \mathbb{R}$ (i.e. periodic of period $1$) with small $C^{2,\alpha}$-norm, we define:
\[C^T_{1+v}=\left\{(x,t)\in\mathbb{R}^{n}\times\mathbb{R}/\mathbb{Z}\, \, :\, \, 0\leq |x|<1+ v\left(\frac{t}{T} \right)\right\}.\]
Such a domain is in fact a small perturbation of the straight cylinder of radius 1, periodic in the vertical direction with period $T$. We look at the problem:
\begin{equation}\label{eq22001}
  \begin{cases}
  \Delta u +f(u)=0 &\mbox{in $C^T_{1+v}$, }\\
  u>0 &\mbox{in $C^T_{1+v}$, }\\
  u=0 &\mbox{on $\partial C^T_{1+v}$,}\\
    \partial_{\nu} u=\mbox{constant} &\mbox{on $\partial C^T_{1+v}$. }
  \end{cases}
\end{equation}
Our aim will be to find a curve $(v,T) =(v(T),T)$, with $v\not\equiv0$, such that problem \eqref{eq22001} has a solution. We shall write it in the equivalent form:

\begin{equation}\label{eq22001 bis}
\begin{cases}
\Delta_{\lambda}\phi+f(\phi)=0 &\mbox{in $C_{1+v}^1$ }\\
\phi>0 &\mbox{in $C_{1+v}^1$ }\\
\phi=0 &\mbox{on $\partial C_{1+v}^1$}\\
|\nabla^{\lambda} \phi|=\mbox{constant} &\mbox{on $\partial C_{1+v}^1$ }
\end{cases}\,,
\end{equation}
where $\Delta_{\lambda} \phi = \Delta_x \phi + \lambda \phi_{tt}$, and $\nabla^{\lambda} \phi= (\nabla_x \phi, \sqrt{\lambda} \phi_t)$.  Indeed, if we set $T = \frac{1}{\sqrt{\lambda}}$, we have that
\begin{equation} \label{formula}
u(x,t) = \phi \left( x, \frac{t}{T} \right)
\end{equation}
is a solution of \eqref{eq22001}.

Since it is clear that (\ref{eq11}) is invariant under translations, it is natural to require that the function $v$ is even. Moreover, sometimes it will be useful to assume the function $v$ has 0 mean. So, we introduce the H\"{o}lder spaces:

\[C_{e}^{k,\alpha}(\mathbb{R}/\mathbb{Z})=\left\{v\in C^{k,\alpha}(\mathbb{R}/\mathbb{Z}): v(-t)=v(t)\right\},\]

\[C_{e,m}^{k,\alpha}(\mathbb{R}/\mathbb{Z})=\left\{v\in C^{k,\alpha}(\mathbb{R}/\mathbb{Z}): v(-t)=v(t),~\int_0^1v \,dt=0\right\}\]
for $k \in \mathbb{N}$.

\medskip

We start with the following result:

\begin{proposition} \label{Pr30}
	Assume that $\lambda>\frac{1}{\overline{T}^2},$ where $\overline{T}$ is given by (\ref{eq266}). Then, for all $v\in C_{e}^{2,\alpha}(\mathbb{R}/\mathbb{Z})$ whose norm is sufficiently small, the problem
	\begin{equation}\label{eq27}
	\begin{cases}
	\Delta_{\lambda} \phi+f(\phi)=0 &\emph{in $C_{1+v}^1$ }\\
	\phi=0 &\emph{on $\partial C^1_{1+v}$}
	\end{cases}
	\end{equation}
	has a unique positive solution $\phi=\phi_{v,\lambda}\in C^{2,\alpha}(C^{1}_{1+v}).$ 
	Moreover, $\phi$ depends smoothly on the function $v,$ and $\phi=\phi_{1}$ when $v\equiv0.$
\end{proposition}
\begin{proof}  Let $v\in C_{e}^{2,\alpha}(\mathbb{R}/\mathbb{Z})$. It will be more convenient to consider the fixed domain $C^{1}_1$ endowed with a new metric depending on $v$. This will be possible by considering the parameterization of $C_{1+v}^{1}$ defined by
	\begin{equation}\label{eq205}
	Y(y,t):=\Big(\big(1+v(t)\big)y,t\Big).
	\end{equation}
	
	
	Therefore, we consider the coordinates $(y,t)\in C^{1}_1$ from now on, and we can write the new metric in these coordinates as
	\[g=\sum\limits_{i}[1+v(t)]^{2}dy_{i}^{2}+\sum\limits_{i}[1+v(t)]v'(t)y_{i}dy_{i}dt+[v'(t)^{2}y^2+1]dt^{2}.\]
	
	Up to some multiplicative constant, we can now write the problem (\ref{eq27}) as
	\begin{equation}\label{eq207}
	\begin{cases}
	\Delta_{\lambda,g} \hat{\phi}+f(\hat{\phi})=0 &\mbox{in $C^{1}_1$ }\\
	\hat{\phi}=0 &\mbox{on $\partial C^{1}_1$}
	\end{cases}
	\end{equation}
	where $\Delta_{\lambda,g}$ is the operator $\Delta_{\lambda}$ rewritten in the metric $g$.
	As $v\equiv0$, the metric $g$ is just the Euclidean metric, and $\hat{\phi}=\phi_{1}$ is therefore a solution of (\ref{eq207}).  In the general case, the expression between the function $\phi$ and the function $\hat{\phi}$ can be represented by
	\[\hat{\phi}=Y^{*}\phi.\]
	For all $\psi\in C_{0,r}^{2,\alpha}(C^{1}_1),$ we define:
	\begin{equation} \label{defN} N(v,\psi):=\Delta_{\lambda,g}(\phi_{1}+\psi)+f(\phi_{1}+\psi).\end{equation}
	We have\[N(0,0)=0.\]
	The mapping $N$ is $C^1$ from a neighborhood of $(0,0)$ in $C_{e}^{2,\alpha}(\mathbb{R}/\mathbb{Z})\times C_{0,r}^{2,\alpha}(C^{1}_1)$ into $C_{r}^{\alpha}(C^{1}_1).$ We point out  that $N$ could fail to be $C^2$ with respect to $v$, since the nonlinearity $f$ is assumed only to be $C^{1,\alpha}$, but in any case it admits the double cross derivative $D_\lambda D_v$. The partial differential of $N$ with respect to $\psi$ at $(0,0)$ is
	\[D_{\psi}N|_{(0,0)}(\psi)=\Delta_{\lambda}\psi+f'(\phi_{1})\psi.\]
	Via the change of variables $w(x,t) = \psi \left( x, \frac{t}{T} \right),$ we can use Corollary \ref{cor} to show that $D_{\psi}N|_{(0,0)}$ is invertible from $C_{0,r}^{2,\alpha}(C^{1}_1)$ into $C_{r}^{\alpha}(C^{1}_1)$. The Implicit Function Theorem therefore yields that there exists $\psi(v,\lambda)\in C_{0,r}^{2,\alpha}(C^{1}_1)$ such that $N(v,\psi(v,\lambda))=0$ for $v$ in a neighborhood of $0$ in $C_{e}^{2,\alpha}(\mathbb{R}/\mathbb{Z})$. The function $\hat{\phi}:=\phi_{1}+\psi$ solves (\ref{eq207}), and moreover the dependence on $\lambda $ is $C^1$.

\end{proof}

\medskip

For any $\lambda > \frac{1}{\overline{T}^2}$ we define the nonlinear operator $G$ as follows. After the canonical identification of $\partial C^{1}_{1+v}$ with $\mathbb{S}^{n-1}\times \mathbb{R}/\mathbb{Z},$ we define the following  operator $G: \mathcal{U}\times(\frac{1}{\overline{T}^2},+\infty) \rightarrow C^{1,\alpha}_{e,m}(\R / \Z)$, where $\mathcal{U}$ is a neighborhood of $0$ in $C^{1,\alpha}_{e,m}(\R / \Z)$, as:
\begin{equation}\label{eq24bis}
G(v,\lambda)(t)= -|\nabla^{\lambda}\phi_{v,\lambda}|_{\partial C^1_{1+v}}+\frac{1}{\mbox{Vol}(\partial C^1_{1+v})}\int_{\partial C^{1}_{1+v}}|\nabla^{\lambda}\phi_{v,\lambda}|,
\end{equation}
where $\phi_{v,\lambda}$ is the solution of (\ref{eq27}) verified by Proposition \ref{Pr30}. Clearly $G$ is a $C^1$ operator, and admits also the crossed derivative $D_\lambda D_v G$ since the operator $N$ defined in \eqref{defN} does.

Clearly, $G$ admits the equivalent expression as the Dirichlet-to-Neumann operator:

\begin{equation}\label{eq24}
G(v,T)(t) =\left.  \partial_{\nu} (u_{v,T})\right|_{\partial C^T_{1+v}} (T\, t) -\frac{1}{\mbox{Vol}(\partial C^T_{1+v})}\int_{\partial C^T_{1+v}} \partial_{\nu}(u_{v,T}),
\end{equation}
where $u_{v,T}$ is related to $\phi_{v,\lambda}$ via the formula \eqref{formula}.
 Notice that $G(v,T)=0$ if and only if $\partial_{\nu} u$ is constant on the boundary $\partial C^T_{1+v}.$ Obviously, $G(0,T)=0$ for all $T<\overline{T}.$ Our goal is to find a branch of nontrivial solutions $(v,T)$ to the equation $G(v,T)=0$ bifurcating from some point $(0, T_{*})$, $T_{*} \in (0, \overline{T})$. To this aim, we will use a local bifurcation argument. This leads to the study of the linearization of $G$ around a point $(0,T)$; this study is the purpose of the next section.

\numberwithin{equation}{section}\section{The linearization of the operator $G$}
\label{Section 5}

We will next compute the Fr\'{e}chet derivative of the operator $G$. For that aim, we will need the following two lemmas.


 \begin{lemma} \label{le32}
Assume that $T<\overline{T},$ where $\overline{T}$ is given by (\ref{eq266}). Then for all $v\in C_{e}^{2,\alpha}(\mathbb{R}/\mathbb{Z})$, there exists a unique solution $\psi_{v,T}$ to the problem
\begin{equation}\label{eq23}
  \begin{cases}
  \Delta\psi_{v,T}+f'(\phi_{1})\psi_{v,T}=0 &\emph{in $C^{T}_1$, }\\
  \psi_{v,T}= \tilde v &\emph{on $\partial C^{T}_1$.}
  \end{cases}
\end{equation}
where $\tilde v(t) := v\left(\frac{t}{T}\right)$.
 \end{lemma}
\begin{proof} Let $\psi_{0}(x,t) \in C^{2, \alpha}(C^{T}_1)$ such that $(\psi_0)|_{\partial C^{T}_1} =  \tilde v$.  
If we set $\omega=\psi_{v,T}{{-\psi_{0}}}$, the problem (\ref{eq23}) is equivalent to the problem
\begin{equation*}
  \begin{cases}
  \Delta \omega+f'(\phi_{1})\omega= - \Big (  \Delta \psi_0+f'(\phi_{1})\psi_0  \Big ) &\mbox{in $C^{T}_1$, }\\
 \omega=0 &\mbox{on $\partial C^{T}_1$.}
  \end{cases}
\end{equation*}
Observe that the right hand side of the above equation is in $C^\alpha_r(C^{T}_1)$. Recall that by Corollary \ref{cor}, $L_D^{T}$ is nondegenerate. Hence it is a bijection and the result follows.

\end{proof}

For the sake of clarity sometimes we will write $\psi_{v}$ instead of $\psi_{v,T}$, when the dependence on $T$ is not relevant.

In next lemma we give some orthogonality results on $\psi_{v}$ defined in Lemma \ref{le32}.
\begin{lemma} \label{Le301}
Let $v\in C_{e,m}^{2,\alpha}(\mathbb{R}/\mathbb{Z})$ and $\psi_{v}\in C_{r}^{2,\alpha}(C^{T}_1)$ be the solution of  (\ref{eq23}). Then
 \[\int_{C^{T}_1}\psi_v z_{j}=0,~\qquad  \int_{\partial C^{T}_1}\partial_{\nu}\psi_v=0\, , \qquad j=1,2,\cdots,l\,.\]
\end{lemma}

\begin{proof}
 We multiply the equation in (\ref{eq202}) by $\psi_v$, the equation in (\ref{eq23}) by $z_{j}$, and integrate by parts to gain
\[\int_{\partial C^{T}_1}\Big(\partial_{\nu} \psi_v \, z_{j}-\partial_{\nu} z_{j} \, \psi_v \Big)=\int_{C^{T}_1}\tau_{j} z_{j}\psi_v\,.\]
Then we can at once gain the first identity by the facts that $z_{j}=0,~ \partial_{\nu} z_{j}$ is constant and $\psi_v= v(\cdot /T)$ has $0$ mean on $\partial C^{T}_1$.

We now define $\kappa\in C_{r}^{2,\alpha}(C^{T}_1)$ as the unique solution of the problem
\begin{equation}\label{eq203}
     \begin{cases}
     \Delta\kappa+f'(\phi_{1})\kappa=0 &\mbox{in $C^{T}_1$, }\\
      \kappa=1 &\mbox{on $\partial C^{T}_1$.}
      \end{cases}
    \end{equation}
whose existence has been verified in Lemma ~\ref{le32} for $T<\overline{T}$. Then we multiply the equation in(~\ref{eq203}) by $\psi_v$, the equation in (\ref{eq23}) by $\kappa$, and integrate by parts to obtain
\[\int_{\partial C^{T}_1}\Big(\partial_{\nu} \kappa \,  \psi_v-\partial_{\nu} \psi_v \, \kappa\Big)=0.\]
Then we can at once gain the second identity by the facts that $\kappa=1, \partial_{\nu} \kappa$ is constant and $\psi_v(x,t)= v\left(\frac{t}{T}\right)$ on $\partial C^{T}_1$.
\end{proof}

\medskip

For $T < \overline{T}$ we can define
the linear continuous operator $H_{T}:C_{e,m}^{2,\alpha}(\mathbb{R}/\mathbb{Z})\rightarrow C_{e,m}^{1,\alpha}(\mathbb{R}/\mathbb{Z})$ by
\begin{equation}\label{eq25}
H_{T}(v) (t) = \left. \partial_{\nu} (\psi_v) \right|_{\partial C_1^T}(T t)  + c \, v,
\end{equation}
where $\psi_v$ is given in Lemma ~\ref{le32} and $c$ is given in \eqref{defc}. We present some properties of $H_T$.

\begin{lemma} \label{Le332}
For any $T<\overline{T},$ the operator
\[H_{T}:C_{e,m}^{2,\alpha}(\mathbb{R}/\mathbb{Z})\rightarrow C_{e,m}^{1,\alpha}(\mathbb{R}/\mathbb{Z})\]
is a linear essentially self-adjoint operator and has closed range. Moreover, it is also a Fredholm operator of index zero.
 \end{lemma}

\begin{proof}
Given $v_i \in C_{e,m}^{2,\alpha}(\mathbb{R}/\mathbb{Z})$, we define $\tilde{v}_i(t)= v_i\left(\frac{t}{T}\right),i=1,2$. Let us compute:
\[\begin{aligned}
 T \left(\int_0^1  H_{T}(v_{1}) {v}_{2}-\int_0^1 H_{T}({v}_{2}) {v}_{1} \right)
&=\int_0^T(\partial_{\nu}\psi_{v_{1}}\tilde{v}_{2}+{ {c}}\tilde{v}_{1}\tilde{v}_{2})-\int_0^T(\partial_{\nu}\psi_{v_{2}}\tilde{v}_{1}+{ {c}}\tilde{v}_{2}\tilde{v}_{1})\\
  & =\int_0^T(\partial_{\nu}\psi_{v_{1}}\tilde{v}_{2}-\partial_{\nu}\psi_{v_{2}}\tilde{v}_{1})\\
    &=\int_0^T(\psi_{v_{2}}\partial_{\nu}\psi_{v_{1}}-\psi_{v_{1}}\partial_{\nu}\psi_{v_{2}})\\
    &={\frac{1}{\omega_n}}\, \int_{C^T_1}(\psi_{v_{2}}\Delta\psi_{v_{1}}-\psi_{v_{1}}\Delta\psi_{v_{2}})\\
    &={\frac{1}{\omega_n}}\, \int_{C^T_1}(f'(\phi_{1})\psi_{v_{2}}\psi_{v_{1}}-f'(\phi_{1})\psi_{v_{1}}\psi_{v_{2}})\\
    &=0.
     \end{aligned}\]
Therefore, we know that the operator $H_{T}$ is self-adjoint.
In addition, the first part of the operator $H_{T}$, the Dirichlet-to-Nenmann operator for $\Delta+f'(\phi_{1})$, is lower bounded since $0$ is not in the spectrum of $\Delta+f'(\phi_{1})$ (see~\cite{AEKS14}).
This yields that there is a constant $c>0$ such that
\[\| v\|_{C^{2,\alpha}(\mathbb{R}/\mathbb{Z})}\leq c\| H_{T}(v)\|_{C^{1,\alpha}(\mathbb{R}/\mathbb{Z})},\]
for all $v$ that are $L^{2}(\mathbb{R}/\mathbb{Z})$-orthogonal to $\mbox{Ker~}(H_{T}).$ It follows that the range of $H_{T}$ is closed.
Therefore, $H_{T}$ is a Fredholm operator of index zero (refer to~\cite{K08}).
\end{proof}

\medskip

We show now that the linearization of the operator $G$ with respect to $v$ at $v=0$ is given by $H_T$, up to a constant.

\begin{proposition} \label{Pr31}
For any $T \in (0, \overline{T})$,
\[
\left.D_{v}(G) \right|_{v=0}= - \phi_1'(1)\, H_{T}\,.
\]
 \end{proposition}

\begin{proof}
	By the $C^1$ regularity of $G$, it is enough to compute the linear operator obtained by the directional derivative of $G$ with respect to $v$, computed at $(v,T)$. Such derivative is given by
\begin{align*}
   G'(w)=\mathop {\lim}\limits_{s\rightarrow 0}\frac{G(sw,T)-G(0,T)}{s}=\mathop {\lim}\limits_{s\rightarrow 0}\frac{G(sw,T)}{s}.
    \end{align*}
    Let $v=sw,$ for $y\in\mathbb{R}^{n}$ and $t\in\mathbb{R},$ we consider the parameterization of $C^T_{1+v}$ given in (\ref{eq205}).
    Let $g$ be the induced metric such that $\hat{\phi}=Y^{*}\phi$ (smoothly depending on the real parameter $s$) solves the problem
    \begin{equation*}
     \begin{cases}
     \Delta_{g}\hat{\phi}+f(\hat{\phi})=0 &\mbox{in $C^T_1$, }\\
      \hat{\phi}=0 &\mbox{on $\partial C^T_1$}.
      \end{cases}
    \end{equation*}
Let $\tilde{w}(t)= w\left(\frac{t}{T}\right)$. We remark that $\hat{\phi}_{1}=Y^{*}\phi_{1}$ is the solution of
\[\Delta_{g}\hat{\phi}_{1}+f(\hat{\phi}_{1})=0\]
in $C^T_1$, and
\[\hat{\phi}_{1}(y,t)=\phi_{1}\big((1+s\tilde{w})y,t\big)\]
on $\partial C^T_1$. Let $\hat{\phi}=\hat{\phi}_{1}+\hat{\psi},$ we can get that
\begin{equation}\label{eq37}
     \begin{cases}
     \Delta_{g}\hat{\psi}+f(\hat{\phi}_{1}+\hat{\psi})-f(\hat{\phi}_{1})=0  &\mbox{in $C^T_1$ },\\
      \hat{\psi}=-\hat{\phi}_{1}  &\mbox{on $\partial C^T_1$}.
      \end{cases}
    \end{equation}
 Obviously, $\hat{\psi}$ is differentiable with respect to $s$. When $s=0$, we have $\phi=\phi_{1}.$ Then, $\hat{\psi}=0$ and $\hat{\phi}_{1}=\phi_{1}$ as $s=0.$ We set
  \[\dot{\psi}=\partial_{s}\hat{\psi}|_{s=0}.\]
 Differentiating (\ref{eq37}) with respect of $s$ and evaluating the result at $s=0$, we have
  \begin{align*}
     \begin{cases}
     \Delta\dot{\psi}+f'(\phi_{1})\dot{\psi}=0  &\mbox{in $C^T_1$, }\\
      \dot{\psi}=-\phi_{1}'(1) \tilde{w}  &\mbox{on $\partial C^T_1$.}
  \end{cases}
    \end{align*}
    Then $\dot{\psi} = -\phi_1'(1)\, \psi_w$ where $\psi_w$ is as given by Lemma \ref{le32} (with $v=w$). Then, we can write
     \[\hat{\phi}(x,t)=\hat{\phi}_{1}(x,t)+s\dot{\psi}(x,t)+\mathcal{O}(s^{2}).\]
    In particular, in a neighborhood of $\partial C^T_1$ we have
   \begin{align*}
   \hat{\phi}(y,t)&=\phi_{1}\big((1+s\tilde{w})y,t\big)+s\dot{\psi}(y,t)+\mathcal{O}(s^{2})\\
         &=\phi_{1}(y,t)+s\big(\tilde{w}r\partial_{r}\phi_{1}+\dot{\psi}(y,t)\big)+\mathcal{O}(s^{2}).
         \end{align*}
In order to complete the proof of the result, it is enough to calculate the normal derivation of the function $\hat{\phi}$ when the normal is calculated with respect to the metric $g$. By using cylindrical coordinates $(y,t)=(rz,t)$ where $r:=|y|>0$ and $z\in\mathbb{S}^{n-1}$, then the metric $g$ can be expanded in $C^{T}_1$ as
         \[g=(1+s\tilde{w})^{2}dr^{2}+2sr\tilde{w}'(1+s\tilde{w})drdt+\big(1+s^{2}r^{2}(\tilde{w}')^{2}\big)dt^{2}+r^{2}(1+s\tilde{w})^{2}\overset{\circ}{h}\]
          where $\overset{\circ}{h}$ is the metric on $\mathbb{S}^{n-1}$ induced by the Euclidean metric. It follows from this expression that the unit normal vector fields to $\partial C^T_1$ for the metric $g$ is given by
          \[\hat{\nu}=\big((1+s\tilde{w})^{-1}+\mathcal{O}(s^{2})\big)\partial_{r}+\mathcal{O}(s)\partial_{t}.\]
   By this, we conclude that
   \[g(\nabla\hat{\phi},\hat{\nu})=\partial_{r}\phi_{1}+s\big(\tilde{w}\partial_{r}^{2}\phi_{1}+\partial_{r}{ {\dot{\psi}}}\big)+\mathcal{O}(s^{2})\]
   on $\partial C^T_1$. From the fact that $\partial_{r}\phi_{1}$ is constant and the fact that the term $\tilde{w}\partial_{r}^{2}\phi_{1}+\partial_{r}\dot{\psi}$ has mean $0$ on $\partial C^T_1$ we obtain
   \[
    G'(w) = \partial_{r}\dot{\psi}(Tt) + \phi_{1}''(1)\, w = - \phi_1'(1)\,  \partial_{r}\psi_w (Tt) + \phi_{1}''(1)\, w = - \phi_1'(1)\, H_T(w).
   \]
   This concludes the proof of the result.
  \end{proof}

\numberwithin{equation}{section}\section{Study of the linearized operator $H_T$}
\label{Section 6}

In view of Proposition \ref{Pr31}, a bifurcation of the branch $(0,T)$ of solutions of the equation $G(v, T)=0$ might appear only at points $(0,T_{*})$ such that $H_{T_{*}}$ becomes degenerate. This will be verified to be true for a precise value $T_{*}<\overline{T}.$ Let us now define the quadratic form associated to $H_T$, namely:
\[ J_{T}: C_{e,m}^{2,\alpha}(\mathbb{R}/\mathbb{Z}) \to \R, \ J_T (v) = \int_0^1 H_T(v) v. \]
We now study the first eigenvalue of the operator $H_{T}$ as
 \[\sigma(T)=\inf\Bigg\{J_{T}(v): v\in C_{e,m}^{2,\alpha}(\mathbb{R}/\mathbb{Z})~,~~\int_0^1v^{2}=1\Bigg\}.\]

\begin{lemma} \label{este}
For any $v\in C_{e,m}^{2,\alpha}(\mathbb{R}/\mathbb{Z}),$
 \[ Q^{T}(\psi_v)= T \omega_n J_{T}(v)\, .\]

\end{lemma}

\begin{proof}
By the divergence formula, we have
\begin{align*}
T \omega_n J_{T}(v)
&=\int_{\partial C^T_1} \psi_v \partial_{\nu}\psi_v +c \int_{\partial C^T_1} (\psi_v)^2 \\
&=\int_{C^T_1}\big(\nabla_{x}\psi_v\nabla_{x}\psi_v+\psi_v\Delta_{x}\psi_v\big) +c \int_{\partial C^T_1} (\psi_v)^2 \\
&=\int_{C^T_1}\big(\nabla_{x}\psi_v\nabla_{x}\psi_v-\psi_v(\psi_v)_{tt}-f'(\phi_{1})\psi_v\psi_v\big) +c \int_{\partial C^T_1} (\psi_v)^2 \\
&=\int_{C^T_1}\big(\nabla\psi_v\nabla\psi_v-f'(\phi_{1})\psi_v\psi_v\big) +c \int_{\partial C^T_1} (\psi_v)^2 \\
&=Q^{T}(\psi_v).
\end{align*}
\end{proof}

Next lemma characterizes the eigenvalue $\sigma(T)$ in terms of the quadratic form $Q^{T}$.
\begin{lemma} \label{le35}
For any $T<\overline{T}$, we have
\[\sigma(T)= \min \Bigg\{Q^{T}(\psi): \psi\in E,~\int_{\partial C^T_1}\psi^{2}=1\Bigg\},\]
where
\begin{equation}\label{eq322}
  E=\Bigg\{\psi\in H^1_r(C^T_1):~\int_{\partial C^T_1}\psi=0,\int_{C^T_1}\psi z_{j}=0,  \ j =1, \dots, l\Bigg\}.
\end{equation}
Moreover, the infimum is attained.
 \end{lemma}
\begin{proof}
Let us define
\begin{equation}\label{eq323}
  \mu:=\inf\Bigg\{Q^{T}_{D}(\psi):\psi\in E,\int_{\partial C^T_1}\psi^{2}= 1 \Bigg\} \in [-\infty, +\infty).
\end{equation}

We first show that $\mu$ is achieved. On that purpose, take $\psi_{n}\in E$ such that $Q^{T}_{D}(\psi_{n})\rightarrow\mu .$

We claim that $\psi_{n}$ is bounded. Reasoning by contradiction, if $\|\psi_{n}\|_{H^1}\rightarrow +\infty$, we define $\xi_{n}=\|\psi_{n}\|_{H^1}^{-1}\psi_{n};$ we can suppose that up to a subsequence $\xi_{n}\rightharpoonup \xi_{0}.$ Notice that $\int_{\partial C^T_1}\xi_{n}^{2}\rightarrow0,$ which yields that $\xi_{0}\in H^{1}_{0,r}(C^T_1).$ We also point out that
\[\int_{C^T_1}f'(\phi_{1})\xi_{n}^{2}\rightarrow\int_{C^T_1}f'(\phi_{1})\xi_{0}^{2} \, , \  \int_{C^T_1} \xi_0 z_j =0, \ j=1, \dots ,l\,.\]
Let us consider the following two cases:\\
\textbf{Case 1:} $\xi_{0}=0.$ In this case
\[Q^{T}_{D}(\psi_{n})=\|\psi_{n}\|^{2}\int_{C^T_1}\big(|\nabla\xi_{n}|^{2}-f'(\phi_{1})\xi_{n}^{2}\big)\rightarrow+\infty\,,\]
which is impossible.\\
\textbf{Case 2:} $\xi_{0}\neq 0.$ In this case
\begin{align*}
\mathop {\liminf}\limits_{n\rightarrow\infty}Q^{T}_{D}(\psi_{n})&=\mathop {\liminf}\limits_{n\rightarrow\infty}
\|\psi_{n}\|^{2}\int_{C^T_1}\big(|\nabla\xi_{n}|^{2}-f'(\phi_{1})\xi_{n}^{2}\big)\\
&\geq\mathop {\liminf}\limits_{n\rightarrow\infty}\|\psi_{n}\|^{2}Q^{T}_{D}(\xi_{0})\,,
\end{align*}
but $Q^{T}_{D}(\xi_{0})>0$ by Prop. \ref{Pr300}. This is again a contradiction.

Thus, $\psi_{n}$ is bounded, so up to a subsequence we can pass to the weak limit $\psi_{n}\rightharpoonup\psi.$ Then, $\psi$ is a minimizer for $\mu$ and in particular $\mu>-\infty.$

By the Lagrange multiplier rule, there exist $\theta_{1},\theta_{2}$ and $\zeta_1, \dots ,\zeta_l$ real numbers so that for any $\rho\in H_{r}^{1}(C^T_1),$
\[\int_{C^T_1}\left(\nabla\psi\nabla\rho-f'(\phi_{1})\psi\rho + \rho \sum_{i=1}^{l} \zeta_i z_i\right)=\int_{\partial C^T_1}\rho((\theta_{1}+c)\psi+\theta_{2}).\]
Taking $\rho=z_j$ above we conclude that $\zeta_j=0$. Moreover, if we take $\rho=\psi$ and $\rho=\kappa$ (given by (\ref{eq203})), we conclude that $\theta_{1}+c=\mu$ and $\theta_{2}=0$, respectively. In other words, $\psi$
 is a (weak) solution of
\begin{equation}\label{eq325}
  \begin{cases}
  \Delta\psi+f'(\phi_{1})\psi=0 &\mbox{in $C^T_1$\,, }\\
  \partial_{\nu}\psi=\mu\psi &\mbox{on $\partial C^T_1$\,.}
  \end{cases}
\end{equation}
By the regularity theory, $\psi\in C_{r}^{2,\alpha}(C^T_1)$. Define $v(t) = \psi|_{\partial C^T_1}(T\, t)$. Observe that:
$$
\int_0^1 v^2 = \frac{1}{T \omega_n}, \ J_T(v)= \frac{1}{T \omega_n} Q^T(\psi)= \frac{1}{T \omega_n} \mu\,.
$$
In the second equality Lemma \ref{este} has been used. After a suitable renormalization we obtain that $\sigma(T) \leq \mu$. But, again by Lemma \ref{este}, the reversed inequality is trivially satisfied, and the proof is concluded.
\end{proof}

\medskip

We are now in conditions to prove the main result of this section:

\begin{proposition} \label{Pr34} Define $T_{*}=\frac{2 \pi}{\sqrt{-\gamma_{1}}}$, where $\gamma_{1}$ is given in \eqref{gamma1}. Observe that by Lemma \ref{lema2.2}, $T^*$ is well defined and $T_{*} \in (0, \overline{T})$. Then:
	\begin{itemize}
		\item[(i)] if $T \in (0,T_{*}),$ then $\sigma(T)>0;$
		\item[(ii)] if $T=T_{*},$ then $\sigma(T)=0;$
		\item[(iii)] if $T \in (T_{*}, \, \overline{T})$ then $\sigma(T)<0.$
	\end{itemize}
	Moreover, $\emph{Ker}(H_{T_{*}})=\mathbb{R}\cos(2\pi t).$ In particular, $\emph{dim Ker}(H_{T_{*}})=1.$
\end{proposition}

\begin{proof} It follows from Lemma \ref{le35} and Proposition \ref{Pr300}, taking into account that $C^{2,\alpha}_{e,m}(\R / \Z)$ contains only even functions.	
\end{proof}

\numberwithin{equation}{section}\section{The bifurcation argument }
\label{Section 7}

In this section, we are in position to prove our main Theorem ~\ref{Th11} by the bifurcation argument. For the sake of completeness, we will now recall the  bifurcation theorem, which is due to Crandall and Rabinowitz. For the proof and for many other applications we refer to ~\cite{K04,S94} and to the original exposition ~\cite[Theorem 1.7]{CR71}.

\begin{theorem} \label{Th401} \textbf{\mbox{(Crandall-Rabinowitz Bifurcation Theorem)}}
	Let $X$ and $Y$ be Banach spaces, and let $U\subset X$ and $\Gamma\subset\mathbb{R}$ be open sets, where we assume $0\in U$. Denote the elements of $U$ by $v$ and the elements of $\Gamma$ by $T$. Let $G:U\times\Gamma\rightarrow Y$ be a $C^{1}$ operator such that
	\begin{itemize}
		\item[i)] $G(0,T)=0$ for all $T\in\Gamma;$
		\item[ii)] $\emph{Ker~} D_{v}G(0,T_{*})=\mathbb{R}\,w$ for some $T_{*}\in\Gamma$ and some $w\in X\setminus\{0\};$
		\item[iii)] $\emph{codim Im~} D_{v}G(0,T_{*})=1;$
		\item[iv)] The cross derivative $D_{T}D_{v}G$ exists and is continuous, and $ D_{T}D_{v}G(0,T_{*})(w)\notin \emph{Im~} D_{v}G(0,T_{*}).$
	\end{itemize}
	Then there is a nontrivial continuous curve
	\begin{equation}\label{eq501}
	s \to (v(s), T(s)) \in X \times \Gamma\, ,
	\end{equation}
	$s \in (-\delta,+\delta)$ for some $\delta>0$, such that $(v(0),T(0))=(0,T_{*})$, $v(s) \neq 0 $ if $s \neq 0$
	and \[G\big(v(s),s\big)=0 \quad for \quad s\in(-\delta,+\delta)\,.\]
	Moreover there exists a neighborhood $\mathcal{N}$ of $(0,T_{*})$ in $X \times \Gamma$ such that all solutions of the equation $G(v,T)=0$ in $\mathcal{N}$ belong to the trivial solution line $\{(0, T)\}$ or to the curve (\ref{eq501}). The intersection $(0,T_{*})$ is called a bifurcation point.
\end{theorem}

Theorem ~\ref{Th11} follows immediately from the following proposition and the Crandall-Rabinowitz theorem.
\begin{proposition} \label{Pr42} The linearized operator $D_{v}G(0,T_{*})$ has a 1-dimensional kernel spanned by the function $w=\cos (2\pi t)$, that is,
	\[\emph{Ker~} D_{v}G(0,T_{*})=\mathbb{R}w.\]
	The cokernel of $D_{v}G(0,T_{*})$ is also 1-dimensional, and
	\[D_{T}D_{v}G(0,T_{*})(w)\notin \emph{Im~} D_{v}G(0,T_{*}).\]
\end{proposition}

\begin{proof} Recall from the Proposition \ref{Pr31}, we know that $D_{v}G(0,T_{*})= - \phi_1'(1) H_{T_{*}}.$ Then we have
	\begin{equation*}
	\mbox{Im~} D_{v}G(0,T_{*})=\mbox{Im~}H_{T_{*}}.
	\end{equation*}
	By the Proposition \ref{Pr34}, 
	we have that the kernel of the linearized operator $D_{v}G(0,T_{*})$ has dimension $1$ and can be spanned by the function $w(t)=\cos (2\pi t)$:
	\[\mbox{Ker~}D_{v}G(0,T_{*})=\mathbb{R}\, w.\]
	Then, $\mbox{codim Im~}(H_{T_{*}})=1$ follows from the fact that $H_{T}$ is a Fredholm operator of index zero by Lemma ~\ref{Le332}.
	
	Here, we are ready to prove $ D_{T}D_{v}G(0,T_{*})(w)\notin \mbox{Im~} D_{v}G(0,T_{*}).$ Taking $\xi\in\mbox{Im~} D_{v}G(0,T_{*})=\mbox{Im~}(H_{T_{*}}),\xi=H_{T_{*}}(v),$ then we have
	\[\int_0^1 \xi w=\int_0^1 H_{T_{*}}(v)w =\int_0^1
	H_{T_{*}}(w)v=0,\]
	because of the fact $H_{T_{*}}(w)=0$. By Lemma \ref{Le332} we have
	\[\mbox{Im~}(H_{T_{*}})=\Bigg\{\xi:\int_0^1\xi w=0\Bigg\}.\]
	Recall that $D_{T}D_{v}G(0,T_{*})(w)=-\phi_1'(1) D_{T}|_{T=T_{*}}H_{T}(w),$ then, in order to prove $D_{T}D_{v}G(0,T_{*})(w)\notin \mbox{Im~} D_{v}G(0,T_{*}),$ we just need to prove that
	\[\int_0^1 \big(D_{T}|_{T=T_{*}}H_{T}(w)\big)w\neq 0.\]
	Actually, by Lemma \ref{este},
	\[\begin{aligned}
	\int_0^1 \big(D_{T}|_{T=T_{*}}H_{T}(w)w\big)
	&=\frac{d}{dT}\Big|_{T=T_{*}}\int_0^1 H_{T}(w)w\\
	&=\frac{1}{\omega_n}\frac{d}{dT}\Big|_{T=T_{*}}\left(\frac{1}{T}Q^{T}(\psi_{w})\right)\\
	&=\frac{1}{\omega_n}\frac{d}{dT}\Big|_{T=T_{*}}\left(\frac{1}{2}Q(\psi_{1})+\frac{2\pi^{2}}{T^{2}}\int_{B}\psi_{1}^{2}\right)\\
	&=-\frac{4\pi^{2}}{\omega_n T_{*}^{3}}\int_{B}\psi_{1}^{2}\neq 0,
	\end{aligned}\]
	where the passage from the second to third equality is given by the following computation of $Q^{T}(\psi)$ with the function $\psi_{w}(x,t)=\psi_{1}(x)\cos\left(\frac{2\pi t}{T}\right)$, with $\psi_1$ as in \eqref{eqlineal}:
	\[
	\begin{aligned}
	Q^{T}(\psi)&=Q^{T}_{D}(\psi)+c\int_{\partial C^T_1}\psi^{2}\\
	&=\int_{C^T_1}\big(|\nabla\psi|^{2}-f'(\phi_{1})\psi^{2}\big)+c\int_{\partial C^T_1}\psi^{2}\\
	&=\int_{C^T_1}\left[|\nabla\psi_{1}|^{2}\cos^{2}\left(\frac{2\pi t}{T}\right)+\psi_{1}^{2}\left(\frac{2\pi}{T}\right)^{2}\sin^{2}\left(\frac{2\pi t}{T}\right)-f'(\phi_{1})\psi_{1}^{2}\cos^{2}\left(\frac{2\pi t}{T}\right)\right]\\
	&\quad+c\int_{\partial C^T_1}\psi_{1}^{2}\cos^{2}\left(\frac{2\pi t}{T}\right)\\
	&= \Big[\int_{B}\big(|\nabla\psi_{1}|^{2}-f'(\phi_{1})\psi_{1}^{2}\big)+c \, \omega_{n}\psi_{1}^{2}(1)\Big]\int_0^T\cos^{2}\left(\frac{2\pi t}{T}\right)\\
	&\quad+\left(\frac{2\pi}{T}\right)^{2}\int_{B}\psi_{1}^{2}\int_0^T\sin^{2}\left(\frac{2\pi t}{T}\right)\\
	&= Q(\psi_{1})\int_0^T\cos^{2}\left(\frac{2\pi t}{T}\right)+\left(\frac{2\pi}{T}\right)^{2}\int_{B}\psi_{1}^{2}\int_0^T\sin^{2}\left(\frac{2\pi t}{T}\right) \\
	&=\frac{T}{2}Q(\psi_{1})+\frac{2\pi^{2}}{T}\int_{B}\psi_{1}^{2}.
	\end{aligned}
	\]
\end{proof}

\medskip

\end{document}